\documentclass{amsart}

\usepackage{graphicx}
\usepackage{hyperref}
\usepackage{amsfonts, amssymb, amsmath, amsthm}
\usepackage[latin1]{inputenc}

\usepackage{tikz}
\usetikzlibrary{arrows.meta, positioning, calc}

\usepackage{amsmath}
\usepackage{amssymb}
\usepackage{amsfonts}
\usepackage{latexsym}
\usepackage{mathtools}
\usepackage{amscd}

\newcommand{\Z}{\mathbb Z}
 
\newcommand{\N}{\mathbb N}

\newcommand{\R}{\mathbb R}

\newcommand{\cok}{\textnormal{cok}} 
\newcommand{\p}{\mathcal P}

\newcommand{\MP}{\mathcal M_{\mathcal P}}
\newcommand{\ME}{\mathcal M^{\textnormal{sq}}}
\newcommand{\MEdelta}{\mathcal M^{\textnormal{sq}, \delta} }
\newcommand{\MPE}{\mathcal M_{\mathcal P}^{\textnormal{sq}} } 
\newcommand{\Mdelta}{\mathcal M^{\textnormal{sq}, \delta} } 

\newcommand{\MPdelta}{\mathcal
  M_{\mathcal P}^{\textnormal{sq}, \delta} }
\newcommand{\MPdeltaess}{\mathcal
  M_{\mathcal P}^{\textnormal{sq}, \delta, \textnormal{ess}} }

\newcommand{\SEP}{\textnormal{SE}_{\mathcal P}}

\newcommand{\slpnz}{\textnormal{SL}_{\mathcal P}(n,\mathbb Z )}
  
\numberwithin{equation}{section}

 \newtheorem{thm}[equation]{Theorem}
 \newtheorem{prop}[equation]{Proposition}
 \newtheorem{lem}[equation]{Lemma}

\theoremstyle{definition}
 \newtheorem{de}[equation]{Definition}
 \newtheorem{rem}[equation]{Remark}

 \newtheorem{qu}[equation]{Question}
 \newtheorem{standingnotation}[equation]{STANDING NOTATION}
\newtheorem{notation}[equation]{Notation}

\title[Shift equivalence implies flow equivalence]{Shift equivalence implies flow equivalence for shifts of finite type}
\author{Mike Boyle}

\address{University of Maryland, College Park MD, U.S.A.}
\email{mmb@math.umd.edu} 

\begin{document}

\begin{abstract} Shifts of finite type  defined from 
  shift equivalent matrices must be flow equivalent.
  \end{abstract}
\maketitle
\tableofcontents

%%%
%%%

\section{Introduction}

A (self)homeomorphism $h: X\to X$ induces a ``suspension'' flow
on its mapping torus. Homeomorphisms $h,h'$ are {\it flow
  equivalent} if there exists a homeomorphism between their mapping
tori which takes flow orbits to flow orbits, respecting the
direction of the flow. Homeomorphisms $h,h'$ are {\it eventually
  conjugate} if for all but finitely many integers $n$, the
homeomorphisms $h^n$ and $(h')^n$ are
conjugate (meaning topologically conjugate, i.e. isomorphic
as topological dynamical systems). 

The purpose of this article is to  prove the following result (in which
$\Z_+$ denotes the semiring of nonnegative integers). 
\begin{thm}\label{maintheorem} 
Let 
  $A,B$ be square  nonnilpotent matrices with entries in $\Z_+$,
defining edge shifts of finite type $\sigma_A$, $\sigma_B.$ 
If  $A$ and $B$ are  shift equivalent over $\Z_+$,
then  $\sigma_A$ and $\sigma_B$ are flow equivalent.
\end{thm} 
The shifts $\sigma_A$ and $\sigma_B$ are eventually conjugate 
if and only if $A$ and $B$ are shift equivalent over
$\Z_+$  \cite[Theorem 7.5.15]{LindMarcus2021},  
and every shift of finite type (SFT) is topologically conjugate
to an edge SFT \cite[Theorem 2.3.2]{LindMarcus2021}.
 Thus Theorem \ref{maintheorem}
can be expressed in another form:

\begin{thm}\label{eventualtheorem}
  For shifts of finite type,
  eventual conjugacy implies flow equivalence.
  \end{thm} 

Conjugacy, eventual conjugacy and flow equivalence are 
 fundamental classifications for shifts of finite type, especially
with regard to associated algebraic invariants. So,
Theorem \ref{eventualtheorem}
is a necessary part of a satisfactory understanding of
shifts of finite type.

Theorem \ref{eventualtheorem} is not vacuous, given 
the Kim-Roush constructions \cite{S21,S11} 
of shifts of finite type which
are eventually conjugate but not conjugate.
Because it is already a nontrivial matter to find subshifts which are
eventually conjugate and not conjugate, it is a
challenge to find obstructions (if they exist) to generalizations
of Theorem \ref{eventualtheorem}.  I have no
example of eventually conjugate subshifts which are not flow
equivalent. It is not difficult to construct nonexpansive systems which
are eventually conjugate but not flow
equivalent (see Section \ref{secExample}).

The validity of 
Theorem \ref{maintheorem} (well known in the irreducible case,
as we discuss later) was remarked  without proof  in my 
2000 expository article
 \cite[Sec. II.9]{BAlgAspects2000}. 
 I have forgotten the ``lost proof''
behind that remark; the proof below is certainly different.
The proof below is based on
an explicit PSE (polynomial shift equivalence)
equation \eqref{pseEquation},
derived from a given shift equivalence over $\Z$.
The PSE equation is an equivalence over $\Z[t]$ of
stabilizations of the matrices
$I-tA$ and $ I-tB$. 
Our proof strategy is straightforward.
For reducible SFTs,  we obtain  a suitable 
partitioned (block) form  PSE equation. 
Setting $t$ equal to 1, we then obtain a 
 partitioned 
$\text{SL}(\Z)$ equivalence of partitioned stabilizations of
$I-A$ and $I-B$, 
 which by existing results 
\cite{BPos2002} 
implies the SFTs $\sigma_A$, $\sigma_B$ are flow
equivalent.

There is a longstanding, mutually beneficial interaction 
 between $C^*$-algebras and symbolic dynamics. 
Flow equivalence has played a part in this,
from the seminal Cuntz-Krieger paper \cite {cuntzkrieger} 
to the landmark  classification paper \cite{errs:complete}
of Eilers, Restorff, Ruiz and S{\o}rensen. 
In this vein, I thank S{\o}ren Eilers for
informing me 
that Theorem \ref{maintheorem}
answers  a question he 
raised at the 2022 Oberwolfach $C^*$-algebras conference:
it follows from the diagram
of implications in \cite[Figure 1, p. 2112]{EilersReport2022} 
that Theorem \ref{maintheorem} 
implies any equivariant stable isomorphism
of not necessarily simple Cuntz-Krieger algebras 
can be replaced
by a diagonal-preserving one.

\thanks{I thank (in chronological order of interactions around
Theorem \ref{maintheorem}) Alfredo Costa,
Kevin Brix, Scott Schmieding, Jeremias Epperlein   
and  Tom Meyerovitch. 
Their  queries, interest and discussions over the years
led to this article. 
In particular, I thank Kevin Brix for the  observation 
mentioned in Remark \ref{brixetc}. Finally, I am grateful to the
two referees for their careful readings of the paper, which led to
consideral umprovement.}

\section{Background}

 Shifts of finite type play
an important role in dynamical systems; 
below,  we only 
establish definitions, notation and some basic facts.
See the standard reference  \cite{LindMarcus2021} for an introduction
to shifts of finite type and additional references.
For an exposition of flow equivalence of subshifts,
see \cite{bce:fei}; for everything algebraic around
the shifts of finite type, see \cite{BSCHStableAlg}. 

In this article, by  a topological dynamical system (or just system)
we mean a selfhomeomorphism of a compact metrizable space,
$T:X\to X$. 
Two such selfhomeomorphisms $S,T$ are
topologically conjugate (isomorphic as topological dynamical systems) if 
there is a homeomorphism $h$ between their domains such
that 
$hS=Th$. They are eventually conjugate if
the homeomorphisms $T^k$ and $S^k$ are conjugate, 
for all but finitely many $k$. (For $k>1$, 
$T^k$ is defined by iteration, $T^k = T\circ T^{k-1}$;  for $k<0$, 
 $T^{-k}= (T^{-1})^k$.) 
Given the homeomorphism $T:X\to X$,
let $Y_T$ denote its mapping torus,  the quotient space 
of $(X\times [0,1])$ under the map identifying
$(x,1)$ and $(Tx,0)$, for all $x$ in $X$.
There is a natural unit speed ``vertical'' flow
($\R$ action) on $Y_T$.
%%% , for which  a real
%%% number $s$ ($s=n+\epsilon$ with $n\in \Z$ and $0\leq \epsilon < 1$)
%%% acts by the homeomorphism 
%%% $ [(x,t)] \mapsto [(T^n(x)x,t+ \epsilon )]$.
Two systems are flow equivalent
if there is a  homeomorphism between their
mapping tori which sends flow lines to flow lines
and is orientation preserving (in the sense that it 
respects the direction of the flow);   
equivalently, the two homeomorphisms are cross sections to a 
common flow.

From a square nonnilpotent  matrix $A$ with entries in  $\Z_+$, one
defines an edge shift of finite type (SFT), $\sigma_A : X_A \to X_A$,
as follows. Let $\mathcal G$ be a directed graph with adjacency matrix
$A$. Let $X_A$ be the set of doubly infinite sequences $x=(x_n)_{-\infty <
  n < \infty}$ such that for all $n$, $x_n$ is an edge of
$\mathcal G$ and
the terminal vertex of $x_n$
equals the initial vertex of $x_{n+1}$. Define
$\sigma_A : X_A \to X_A$ by the rule  $(\sigma_Ax)_n = x_{n+1}$.
With the natural topology,  $X_A$ is a compact metrizable zero dimensional
space, and $\sigma_A$ is a homeomorphism. 

Let $\mathcal S$ denote a semiring containing 0 and 1.
A matrix over a set $\mathcal S$ is  a matrix with every
entry in $\mathcal S$. A matrix equation over $\mathcal S$
is  an equation of matrices with all entries in $\mathcal S$.
A nonnegative matrix is a matrix with every entry a nonnegative number. 
Square matrices
$A,B$ are shift equivalent over $\mathcal S$ ($\text{SE}-\mathcal S$)
if there is a positive integer $\ell$ (the ``lag'')
and matrices $R,S$ over $\mathcal S$ such that
the following equations hold:
\[
A^{\ell} = RS\ ,\quad B^{\ell}=SR\ , \quad\quad
AR = RB\ , \quad BS = SA\ .
\]
Above, the square
matrices $A,B$ can have different sizes, and then the matrices 
$R,S$ will not be square. 
For an example, set 
 $\mathcal S=\Z_+$, $\ell =1$, 
$A=\left(\begin{smallmatrix} 2 \end{smallmatrix}\right)$, 
$B=\left(\begin{smallmatrix} 1&1\\1&1 \end{smallmatrix}\right)$,   
$R=\left(\begin{smallmatrix} 1&1 \end{smallmatrix}\right)$,   
$S=\left(\begin{smallmatrix} 1\\1 \end{smallmatrix}\right)$.  
% $B=\begin{pmatrix} 1&1\\1&1 \end{pmatrix}$,
% $S=\begin{pmatrix} 1\\1 \end{pmatrix}$,
% $A=\begin{pmatrix} 2 \end{pmatrix}$,
% $B=\left(\begin{smallmatrix} 1&1\\1&1 \end{smallmatrix}\right)$,  
% $R=\begin{pmatrix} 1&1 \end{pmatrix}$,

$\text{SE}-\mathcal S$ is an equivalence relation on square matrices
over $\mathcal S$. An elementary strong shift equivalence over
$\mathcal S$ ($\text{ESSE}-\mathcal S$) is a
lag 1 shift equivalence.  
Strong shift equivalence over $\mathcal S$
($\text{SSE}-\mathcal S$) is the transitive closure of 
$\text{ESSE}-\mathcal S$. For square nonnilpotent matrices
$A,B$ over $\Z_+$,
$\sigma_A$ and $\sigma_B$ are conjugate iff $A$ and $B$ are
$\text{SSE}-\Z_+$;  
$\sigma_A$ and $\sigma_B$ are eventually conjugate iff $A$ and $B$ are
$\text{SE}-\Z_+$.  

Let $\mathcal R$ be a ring. 
The  general linear group 
$\text{GL}(n,\mathcal R)$ is the group of $n\times n$ matrices $U$
over $\mathcal R$ for which there exists a matrix $V$ over $\mathcal R$
such that $UV=VU=I$. 
The elementary group
$\text{El}(n,\mathcal R)$ is the subgroup of
$\text{GL}(n,\mathcal R)$ generated by 
% transvections (by a transvection, we mean
% a square matrix
matrices equal to the identity
except possibly in a single off-diagonal entry.  
If $\mathcal R$ is commutative,
then $\text{SL}(n,\mathcal R)$ is the subgroup of
$\text{GL}(n,\mathcal R)$ consisting
of the matrices with determinant 1.
It is well known that $\text{SL}(n, \Z) =
\text{El}(n, \Z)$.

The infinite general linear group $\text{GL}(\mathcal R)$ is
the direct limit group 
\[
\text{GL}(n,\mathcal R) \to \text{GL}(n+1,\mathcal R)
\to \text{GL}(n+2,\mathcal R) \to \cdots
\]
with the bonding maps given by
% inclusions
the rule 
$C\mapsto \left( \begin{smallmatrix} C&0\\0&1 \end{smallmatrix}\right)$.
The infinite groups $\text{El}(\mathcal R)$ and
$\text{SL}(\mathcal R)$ are defined as direct limits
% by inclusions
in
the same way. For any square matrix $C$
and identity matrix $I$, we refer to a matrix
$ \left( \begin{smallmatrix} C&0\\0&I \end{smallmatrix}\right)$ as a
stabilization of $C$. 

Matrices $C,D$ are $\text{SL}(n,\mathcal R)$ equivalent if there are
matrices $U,V$ in $\text{SL}(n,\mathcal R)$ such
that $UCV=D$. We may sometimes commit abuse of notation
of the following sort,
writing 
$\text{SL}(\mathcal R)$ in place of
$\text{SL}(n,\mathcal R)$.

\begin{prop} \label{diagblocksirred} 
  Every nonnilpotent square matrix over $\Z_+$ is SSE-$\Z_+$ to
  a block upper triangular  matrix with every diagonal block
  irreducible.
\end{prop}
\begin{proof} The proof is an exercise
  (see e.g. \cite{BoyleJordanForm1984}).
  From a block upper triangular form with irreducible and zero diagonal blocks,
 iteratively one uses an ESSE to eliminate zero diagonal blocks. E.g., 
  \begin{align*}
    \begin{pmatrix} A & B & C \\ 0 & 0 & D \\ 0&0&E \end{pmatrix}
    &=     \begin{pmatrix} I & C \\  0 & D \\ 0&E \end{pmatrix}
    \begin{pmatrix} A & B & 0 \\  0 &0 & I  \end{pmatrix}\ , \\
        \begin{pmatrix} A  & AC+ BD \\ 0 & E  \end{pmatrix}
        &=    \begin{pmatrix} A & B & 0 \\  0 &0 & I  \end{pmatrix}
 \begin{pmatrix} I & C \\  0 & D \\ 0&E \end{pmatrix}
        \ .
        \end{align*} 
      \end{proof}

\begin{notation}
The rows and columns of an $m\times n$ matrix are indexed by totally ordered
sets (typically, $\{ 1, \dots , m\}$ and 
$\{ 1, \dots , n \}$).
For arguments in this paper, it will be convenient later 
to use a slightly more general definition of \lq\lq matrix\rq\rq\ 
(allowing  finite index sets without orderings).
We state this definition at 
 the beginning of Section \ref{partitionedSection},
and use it from that point on in the paper. 
We use  the
       Kronecker delta notation:  
     $\delta_{ij}= 0 $ if $i\neq j$,  and $\delta_{ij} =1 $ if
       $ i= j$.
We  use          
 $\sqcup$ to represent the disjoint union of sets. 
\end{notation}

\section{The PSE  equation}

From a given shift equivalence over a ring $\mathcal R$,  
we derive an explicit equivalence of  matrices over
the polynomial ring $\mathcal R[t]$.

\begin{thm} \label{keyprop}
Suppose $\mathcal R$ is a ring and matrices $R,S$ give a shift equivalence
over $\mathcal R$ from $A$ to $B$ with lag $\ell$:
\[
A^{\ell} =RS,\quad B^{\ell}=SR, \quad AR = RB, \quad SA=BS \ .
\]
Then 
\[
  \begin{pmatrix} I-tA & -R \\ 0 & I
  \end{pmatrix}
  \begin{pmatrix} I + (tA) + \cdots +(tA)^{\ell -1} & R \\
    -t^{\ell}S  & I-tB
  \end{pmatrix}
\     = \ 
    \begin{pmatrix} I& 0 \\ -t^{\ell}S  & I -tB 
    \end{pmatrix} \ .
    \]
Let $U$ be the intertwining matrix, 
    $U =   \begin{pmatrix} I + (tA) + \cdots +(tA)^{\ell -1} & R \\
  -t^{\ell}S & I-tB  \end{pmatrix}      $; 
then the following PSE (polynomial shift equivalence) equation holds: 
\begin{equation} \label{pseEquation}
  \begin{pmatrix} I & 0 \\ t^{\ell}S & I
  \end{pmatrix}
  \begin{pmatrix} I & -R \\ 0 & I
  \end{pmatrix}
    \begin{pmatrix} I-tA & 0 \\ 0 & I
    \end{pmatrix}
    U
    \     = \ 
    \begin{pmatrix} I& 0 \\ 0  & I -tB 
    \end{pmatrix}
    \end{equation} 
If $\ell =1$, then $U$ 
lies in $El(\mathcal R[t])$.
In any case, $U$ lies in $GL(\mathcal R[t])$. 
If $\mathcal R$ is an integral domain, 
then $U\in SL(\mathcal R[t])$.

In particular, if $\mathcal R = \Z$ then 
$U\in SL(\Z [t])$, and the equation \eqref{pseEquation} 
gives a stabilized $SL(\Z[t])$-equivalence of
$(I-tA)$ and $(I-tB)$.

\end{thm}
\begin{proof}
  The validity of the equation is verified by a  computation.
  It remains to verify the claims for $U$.
  If $\ell =1$, then $U = \begin{pmatrix} I & R \\ -tS & I-tB\end{pmatrix}$
      and
    \[
    \begin{pmatrix} I & R \\ -tS & I-tB\end{pmatrix}
    \begin{pmatrix} I & -R \\ 0 & I\end{pmatrix}
      =
      \begin{pmatrix} I & 0 \\ -tS & I\end{pmatrix} \ .
      \]

      Now suppose $\ell > 1$ and define the matrix
      \[
      V = \begin{pmatrix}
        I-tA & -R \\
        t^{\ell}S & I +(tB) + \cdots +(tB)^{\ell -1}
      \end{pmatrix} \ . 
      \]
      A computation shows $UV=VU=I$, so $U\in GL(\mathcal R)$.

      If $\mathcal R$ is commutative, then $\det (U)$ lies in the
      units group of $\mathcal R[t]$. If $\mathcal R$ is an integral
      domain, then the units group of $\mathcal R[t]$ equals 
      the units group of $\mathcal R$; in this case,
       $\det U$ equals the evaluation
      of $\det U$ at $t=0$, which by \eqref{pseEquation}
      must be 1. 
\end{proof}

\begin{rem}
  For definiteness, in this paper
  we simply refer to \eqref{pseEquation} as the 
  PSE equation. 
However, if $0\leq i \leq \ell$, then \eqref{pseEquation} 
 remains
true if $t^lS$ and $R$ are replaced by
$t^{\ell -i}S$ and $t^iR$. 
One might consider such a modification,
or any obvious   rearrangement of \eqref{pseEquation},  
 as \lq\lq the\rq\rq\  PSE equation, or \lq\lq a\rq\rq\ PSE equation. 
We use \lq\lq PSE equation\rq\rq\   by analogy to the
PSSE (polynomial strong shift equivalence) equations
derived in \cite{BW04} from a strong shift equivalence. 
\end{rem}

\begin{rem}
  As originally
  pointed out by Wagoner,   shift equivalence over a ring $\mathcal R$
   of matrices $A,B$ is equivalent to 
  isomorphism of the cokernel $\mathcal R[t]$-modules of $(I-tA)$
  and $(I-tB)$ (see \cite[Sec. 4.3]{BSCHStableAlg}
 for a thorough discussion 
  of this). From the
  cokernel isomorphism,
  and the injectivity of $(I-tA)$ and $ (I-tB)$, 
  Fitting's argument in \cite{Fitting1936} (as presented by Warfield
  in \cite{Warfield1978}) gives an equivalence
  of the stabilized versions of $(I-tA)$   and $(I-tB)$.
  The explicit equivalence \eqref{pseEquation} comes from
  plugging into the Fitting argument explicit cokernel isomorphisms derived
  from the shift equivalence equations.
  It is already known that shift equivalence of
  $A,B$ gives  equivalences of stabilizations
  of $(I-tA)$ and $(I-tB)$ (see 
\cite{BoSc1} or \cite{BSCHStableAlg}). 
The  explicit 
  equivalence \eqref{pseEquation}  is computed 
  for use in proving Theorem \ref{maintheorem}. 
  \end{rem}

To prove Theorem \ref{maintheorem}, we  will
need  a  version of
Theorem \ref{keyprop} for partitioned matrices.
We develop the relevant material on partitioned matrices in
the next section.

\section{Partitioned matrices} \label{partitionedSection}

\begin{standingnotation} \label{standingnotation}
For the rest of the paper,  a  \lq\lq matrix\rq\rq\  $M$ 
will come  with 
finite nonempty sets 
 $\mathcal I^M$ and $\mathcal J^M$ 
indexing its rows and columns. 
In a slight abuse of
notation  
we will not require 
these index sets  
to be  
the standard integer intervals,
or even be ordered sets.\begin{footnote}{This
abuse of notation lets us avoid superfluous
complicating bijections from convenient index sets to sets of 
the 
form $\{ 1, ... , m\}$. The choice of such bijections -- to produce
standard matrices -- isn't relevant for us, because conjugation
by a permutation matrix respects flow equivalence and shift
equivalence.}\end{footnote}
A \lq\lq standard matrix\rq\rq\ is a matrix
is (for some $m,n$) an $m\times n$ matrix $M$ 
with the usual totally ordered index sets,
$\mathcal I_M=\{1, \dots, m\}$ and
$\mathcal J_M=\{1, \dots, n\}$.
\end{standingnotation} 

 We state now  some 
(more or less obvious) definitions
we use to  accommodate our slightly expanded
definition of matrix. We use $\R_+$ to denote the semiring
of nonnegative real numbers. 
A matrix $M$ is {\it square} if $\mathcal I_M= \mathcal J_M$. 
An {\it irreducible} matrix is a square matrix $A$ over $\R_+$
such that for every index pair $(i,j)$ there exists $k>0$ such that
$A^k(i,j)>0$.
A {\it principal} submatrix of a square matrix $A$ is the restriction of $A$
to $\mathcal I \times \mathcal I$, for some nonempty subset of
$\mathcal I^A$. 
 An {\it irreducible component} of 
a square matrix  is 
 a principal submatrix which is irreducible and is not properly
 contained in another irreducible principal submatrix.
 An {\it essentially irreducible} matrix
 is a square matrix  with a unique irreducible
 component. An {\it essentially cyclic} matrix is an essentially
 irreducible matrix whose irreducible component is a cyclic permutation
 matrix.

Throughout, $\p$ denotes a given finite poset; $\preceq$ denotes its
transitive reflexive antisymmetric relation; and $\prec $ denotes
the transitive
relation $\preceq $ and $\neq $. 
A {\it $\p$-partitioned matrix} is a  matrix $A$  together
with partitions
$
% \mathcal I^A =
\{\mathcal I^A_p: p \in \p\}$ of $\mathcal I^A$ and
$
% \mathcal J^A =
\{\mathcal J^A_p : p\in \p\}$  of $\mathcal J^A$ 
 into nonempty
sets
such
for $r,s$ in $\p$, 
\[
%i \in \U_r,\  j\in \V_s \text{ and } A(i,j) \neq 0
i \in  I_r^A ,\  j\in J_s^A \text{ and } A(i,j) \neq 0
\implies r\preceq s\ .
\]
% Formally,
% a $\p$-partitioned matrix is the triple 
% $\mathcal A =(A,\mathcal I, \mathcal J)$.
A $\p$-partitioned matrix $A$ has a $\p $-block structure,
where the $(p,q)$ block $A\{p,q\}$ is the restriction of $A$ to
the index set
% $\mathcal I_p \times \mathcal J_q$
$\mathcal I^A_p \times \mathcal J^A_q$
  \begin{footnote}
        {Fix  an ordering $ p_1, \dots , p_N $ of the elements of
      $\p$ such that $p_i \prec p_j \implies i<j $. Choose  bijections 
      $\mathcal I^A \to \{1, \dots , m\}$ and
      $\mathcal J^A \to \{1, \dots , n\}$ 
      such that $\mathcal I^A_1 , \dots \mathcal I^A_N$
and  $\mathcal J^A_1 , \dots \mathcal J^A_N$ 
      become successive
integer intervals.  
 With this ordering
      of indices, $A$ presents as a standard $m\times n$ 
      block upper triangular matrix, with rectangular blocks
    $A\{p,q\}$; $A\{p,q\}$ is a zero block if $p\not\preceq q$.}  
    \iffalse
        {Fix  an ordering $ p_1, \dots , p_N $ of the elements of
      $\p$ such that $p_i \prec p_j \implies i<j $. Choose  bijections 
      $\mathcal I^A \to \{1, \dots , m\}$ and
      $\mathcal J^A \to \{1, \dots , n\}$ 
      such that $\mathcal I^A_1 , \dots \mathcal I^A_N$
and  $\mathcal J^A_1 , \dots \mathcal J^A_N$ 
      become successive
integer intervals.  
Then our blocks present as traditional matrix blocks and the
block structure is upper triangular. }
        \fi
  \end{footnote}.  When $p=q$, we may refer to $A\{p,q\}$ 
  as a diagonal block.

  Let $\MP (S)$ denote the set of 
  $\p$-partitioned matrices with entries in $S$. 
We  always assume $S$ is contained in a semiring,  with additive and
multiplicative identity elements $0$ and $1$. 
For $A,B$ in $\MP (S)$,
we declare $A+B$ to be  defined in
$\MP (S)$ iff $\mathcal I^A = \mathcal I^B$ and
$\mathcal J^A = \mathcal J^B$; then (of course)
$\mathcal I^{A+B} = \mathcal I^A$ and
$\mathcal J^{A+B} = \mathcal J^A$. 
We declare the product $AB$ to   be
defined in $\MP (S)$  iff
$\mathcal J_A = \mathcal I_B$, in which case
we define $\mathcal I^{AB} =  \mathcal I^A$ and 
$\mathcal J^{AB}= \mathcal J^B$.
Here, to check $AB$ is $\p$-partitioned, note 
\begin{align*}
  (AB)\{p,q\}  = \sum_{r}  A\{p,r\} B\{r,q\}  
 = \sum_{r: \, p \preceq r \preceq q} A\{p,r\} B\{r,q\} \ .
\end{align*}
Thus $(AB)\{p,q\}\neq 0$ implies  there exists $r$ such that
$p \preceq r \preceq q$, and therefore $p\preceq q$.

A matrix $A$ in $\MP(\mathcal S)$ is a square $\mathcal P$-partitioned
matrix if $\mathcal I^A_p=\mathcal J^A_p$ for each $p$ in $\mathcal P$. 
We define $\MPE(\mathcal S)$ to be the set of 
square $\mathcal P$-partitioned matrices over $\mathcal S$;
so, $A^2$ is defined in $\MP(\mathcal S)$ iff $A\in \MPE(\mathcal S)$.
The set of $n\times n$ matrices in $\MP(\mathcal S)$ is
denoted $\MP(n,\mathcal S)$. 

Associated to a square  matrix $A$ over $\Z_+$ is its poset $\mathcal P^A$
of irreducible components, in which 
$p\preceq q$ if and only if
for some $k$ there is an index $i$ for component $p$ and
an index $j$ for
component $q$ such that $A^k(i,j)>0$. This is the poset which
is meaningful for the shift of finite type $\sigma_A$. 
For $A$ in $\MPE(\mathcal \Z_+)$,  when  each $A\{ p,p\}$ is essentially
irreducible, there is a bijection $\mathcal P\to \mathcal P^A$
sending $p$ to the irreducible component of $A\{p,p\}$;
 with this correspondence, we simplify notation by letting  
  $\p$ be the set of  irreducible components. 
However,  in general 
the poset relations of $\p$ and $\p^A$ can be different,
because   $A\{ p, q\} \neq 0$ need not  
imply $p \preceq q$ in $\p^A$. For example, consider the standard
square matrix 
$A =\left(\begin{smallmatrix} 1&1&0\\ 0&0&0 \\ 0&0&1\end{smallmatrix}\right)$,
as an element of $\MPE (\Z_+)$ with 
 $\p$-partition sets $\mathcal I_p=\{ 1\}, \mathcal I_q=\{ 2,3\}$. 
Then $A\{p,q\} \neq 0$, so $p\prec q$ in $\p$, but $p\nprec q$ in
$\p^A$. 

So, we consider  a class in $\MPE$ more tightly related to the irreducible
components poset. We define 
$\MPdelta (\Z_+)$ 
to be the set of
$A$ in   $\MPE(\Z_+)$ such that 
for all $p,q$ in $\mathcal P$ 
the following hold.   
\begin{enumerate} 
\item 
  $A\{ p,p\}$ is an  irreducible component of $A$.
 \item 
 $p\preceq q$ in $\p$ if and only if 
   there exist $k>0$
   % and indices $i,j$ from indices of the
%  irreducible components in $A\{p,p\}$ and $B\{q,q\}$
  such that $A^k\{p,q\} \neq 0$ .
\end{enumerate}
Then, on account of condition (1), for both $\p$ and $\p^A$ 
the poset relation $\prec$  is generated by
the relation $  \prec \prec  $,
where $ p \prec \prec  q$ when $A\{p,q\} \neq 0$.  
Thus for  $A \in \MPdelta (\Z_+)$,
the poset $\p$ can be identified with $\p^A$.

Now let $\MEdelta(\Z_+)$ be the set of square matrices over $\Z_+$
for which every index hits an irreducible component. 
Given   $A$ in $\MEdelta(\Z_+)$, 
 for  $\p=\p^A$  define the natural $\p$-partitioning of indices: 
 $\mathcal I_{\p} = \mathcal J_{\p}$
is the set of indices through
irreducible compoent $\p$.
This is the   
unique\begin{footnote}
{Any square matrix $A$ over $\Z_+$ has a 
  partition of indices indexed by $\mathcal P^A$ 
  for which each $A\{p,p\}$ is essentially
  irreducible. However,  this partition is not unique if
  $A\notin \MEdelta(\Z_+)$.}
%% Condition
%%    (1) in the definition of $\MPdelta$ lets us avoid that
%%  complication in the passage from shift equivalence to flow equivalence.
%%  On the other hand, to apply a flow equivalence classification result, 
%%  we will later have to consider matrices in $\MPdeltaess (\Z_+)$.} 
\iffalse
{For 
        nonnegative matrices
    with zero as well as irreducible diagonal blocks, we lose uniqueness
    of compatible $\p^A$ partitions. Condition
    (1) in the definition of $\MPdelta$ lets us avoid that
    complication.}
\fi
  \end{footnote}
$\p$-partitioning of indices with  which  $A$
is an element of  $\MPdelta$. 

For examples, we list three standard square matrices,
with rows and columns indexed by integer intervals $\{1, \dots , m\}$:
\[
A=\begin{pmatrix} 5&1&0&0 \\ 0&1&1&0 \\ 0&0&0&3 \\ 0&0&0&7
\end{pmatrix} \, \quad \quad 
B=\begin{pmatrix} 5&1&0 \\ 0&1&3 \\0&0&7
\end{pmatrix}
 \, \quad
 C=\begin{pmatrix} 3&0&1&0  \\ 0&1&1&0 \\ 0&0&0&1 \\
 0&0&0&5
 \end{pmatrix}
\ .
\]
Each matrix has three irreducible components $p_1,p_2,p_3$ (indexed
in the order of the corresponding diagonal blocks).
  Of the three matrices,
only $B$ is in 
  $  \MEdelta(\Z_+)$. The unique partitioning
  of indices with which $B$ becomes
  an element of $\MPdelta(\Z_+)$ is
  $\{1\}, \{2\}, \{3\}$.
  For
$A$ and $B$, $p_i \prec p_j$ iff $i<j$.
  The matrices $A,B$ define topologically
  conjugate edge SFTs, because
\[
  A=\left(\begin{smallmatrix} 1&0&0 \\ 0&1&0 \\ 0&0&3 \\ 0&0&7
  \end{smallmatrix}\right)
  \left(\begin{smallmatrix} 5&1&0&0 \\ 0&1&1&0 \\ 0&0&0&1 
  \end{smallmatrix}\right)
  \quad  \textnormal{ and } \quad 
    B=\left(\begin{smallmatrix} 5&1&0&0 \\ 0&1&1&0 \\ 0&0&0&1 
  \end{smallmatrix}\right)
  \left(\begin{smallmatrix} 1&0&0 \\ 0&1&0 \\ 0&0&3 \\ 0&0&7
  \end{smallmatrix}\right) \ . 
\]
 For the matrix $C$,
 $p_i \prec p_j$ iff $j=3$ and $i\in \{1,2\}$.
  With  the partitioning of indices by 
  $\{1\}, \{2,3\}, \{4\}$, 
  the submatrix $C\{p_1, p_2\}$ is nonzero, even though
  $p_1 \not\preceq p_2$ in $\p^C$; with respect to the partitioning
  by $\{1\}, \{2\}, \{3,4\}$, though,  
$C\{p_i,p_j\} \neq 0$ does imply 
  $p_i \preceq p_j$ in $\p^C$, and the posets $\p$, $\p^C$ are the same.

With the exception of $\text{SL}_{\p}(n, \Z)$ equivalence, which
has an additional condition on diagonal blocks 
(see Definition \ref{defSLPEquiv}), in this paper 
a matrix relation with a modifer $\p$
is the same  relation on the class of $\p$-partitioned matrices
(i.e., achieved by operations compatible with the 
$\p$-partition structure). 
For example, $ A $ and $ B$ are
$\SEP-\Z_+$ if there are $R,S$ in $\MP(\Z_+)$
and a positive integer $\ell$ 
such that the following equations of $\p$-partitioned matrices hold:
\begin{equation} \label{seequation}
 A^{\ell}  =  RS\ , \quad B^{\ell} = SR \ , \quad \quad
AR = RB \ , \quad  BS = SA\ . 
\end{equation}

\begin{prop} \label{fromSEtoSEP}
  Suppose   $A,B$  are standard matrices in
$\MEdelta(\Z_+)$, and 
  there is a shift equivalence
  over $\Z_+$,
\begin{equation}\label{SEequation}
 A^{\ell}  =  RS\ , \quad B^{\ell} = SR \ , \quad \quad
AR = RB \ , \quad  BS = SA\ . 
\end{equation} 

Then the following hold.
\begin{enumerate}
  \item 
    The posets $\mathcal P^A$ and $\mathcal P^B$ are isomorphic.
    \item 
With $\mathcal P=\mathcal P^A$, there is a 
 unique choice of 
 $\p$-partitions for 
$B,R,S$ such that the equation \eqref{SEequation} defines
an $SE_{\p }-\Z_+$ with $A,B$ in $\MPdelta$.
\end{enumerate} 
\end{prop}

\begin{proof}
  Because $A \in \mathcal M^{\delta}(\Z_+)$, there is a unique
  partitioning of indices with $\mathcal I^A = \mathcal J^A$ with
  respect to which $A$ is a $\mathcal P^A$-partitioned matrix.
  The analogous statement likewise holds for $B$.
  Define the unique  index partitions for $R$ and $S$
  which could be compatible (for multiplication of partitioned
  matrices): 
  $\mathcal I^R=\mathcal J^S = \mathcal I^A$ and
  $\mathcal I^S=\mathcal J^R = \mathcal I^B$. 
We make the following claims. 
  \begin{enumerate}
  \item
    (i) %   Given
    For $p$ in $\p^A$, there is a unique  $\tilde p$ in
    $\p^B$ such that $R\{p,\tilde p\}S\{\tilde p,p\} \neq 0$; 

    % given
    for $q$ in $\p^B$, there is a unique  $\tilde q$ in
    $\p^A$ such that $S\{q ,\tilde q\}R\{\tilde q,q\} \neq 0$. \\    
  (ii)  $(A^{\ell})\{p, p\} = R\{p, \tilde p\}S\{\tilde p, p\}$
  and
  $(B^{\ell})\{\tilde p, \tilde p\} = S\{\tilde p, p\} R\{p, \tilde p\}$.\ \\ 
  (iii)  For $p\in \p^A$ and $q\in \p^B$:
   $\tilde p =q \iff  \tilde q= p$.

\item The map $p \mapsto \tilde p$ defines a poset isomorphism
  $\p^A \to \p^B$.

\item If $R\{ p, \tilde r\} \neq 0 $ or
  $S\{ \tilde p, r\} \neq 0$, then $ p \preceq r$ in $\mathcal P^A$. 
  \end{enumerate}
  Given the claims, set $\p = \p^A$ and  use (2) to replace the
  $\p_B$ partitions  \eqref{SEequation}  with $\p$ partitions such that
by (2) and (3),   \eqref{SEequation} becomes an equation of $\p$-partitioned
  matrices. 

   It remains to prove the claims.

   (1)(i)  By symmetry, it suffices to prove the first half. 
   Suppose $j\in  \mathcal I^B_q$ and $k\in \mathcal I^B_r$ and there
   are $ i_1, i_2, i_3, i_4$ in $\mathcal I^A_p$ with 
$   R(i_1, j)S(j, i_2) > 0$  and 
   $   R(i_3, k)S(k, i_4) > 0$.
   Because $\{i_2, i_3\} \subset \mathcal I^A_p$,
   there exists $m>0$ such that
   \[
   R(i_1, j)S(j, i_2)A^m(i_2,i_3)R(i_3, k)S(k, i_4) > 0 
   \]
   and therefore
   $
   B^{m+\ell} (j,k) = (SA^mR)(j,k) > 0 
   $.
   Likewise, there exists $m$ such that
   $
   B^{m+\ell} (k,j) = (SA^mR)(k,j) > 0 
   $.
      Therefore,  $q=r$. 

      (ii)       Because $RS=A^{\ell}$ and $SR=B^{\ell}$, these equalities
   follow from (i). 

   (iii)   Suppose  $p\in \p^A$ and $q=\tilde p$.  Then 
   \[
   0 \neq   A^{\ell}\{p,p \} A^{\ell}\{p, p\}
   = R(p, q\} S\{ q, p\}R(p, q\} S\{q, p\} \ , 
   \]
hence  $S\{q, p\}R(p, q\} \neq 0$, and therefore 
 $\tilde q =p$. Similarly, $\tilde q = p  \implies q = \tilde p$. 
     
   (2) By (iii),  $p\mapsto \tilde p$ defines a bijection $\p^A \to \p^B$.
   Suppose  $p \preceq r$ in $\p^A$; 
we will prove $\tilde  p \preceq \tilde r$ in $\p^B$. 
   Because $A \in \MPdelta$,
   we have $m>0$ such that 
   \begin{align*}
     0 \neq   A^{\ell}\{p,p\}A^m\{p,r\}A^{\ell}\{r,r\} 
     =   R\{p, \tilde p\}S\{\tilde p, p\}A^m\{p,r\}
     R\{ r,\tilde r\}S\{\tilde r,r\}  
   \end{align*}
   and therefore
   \[
   B^{\ell +m}\{\tilde p, \tilde r\} = 
   (SA^mR)\{\tilde p, \tilde r\} \neq 0 \ .
   \]
   This implies $\tilde  p \preceq \tilde r$ in $\p^B$.
   Similarly, $q \preceq s$ in $\p^B$ implies
   $\tilde  q \preceq \tilde s$ in $\p^A$.

   (3) Suppose $R\{p, \tilde r \} \neq 0$.
   Then
   \begin{align*} 
     0  \neq   
     R\{p, \tilde r\}B^{\ell}\{\tilde r, \tilde r \}
     &   =
R\{p, \tilde r\}S\{ \tilde r, r\} R\{ r, \tilde q \} 
\\ 
&   \leq    A^{\ell} \{ p,r\} 
 R\{q, \tilde q\} \ .
   \end{align*} 
   It follows that  $A^{\ell}\{p,r\}\neq 0$, and therefore 
   $p \preceq r$. Similarly,
   $S\{\tilde p, r \} \neq 0$ implies
   $0\neq B^{\ell}\{\tilde p, \tilde r\}\neq 0$, hence 
 $p \preceq r$. 

     \end{proof}

  \begin{de} 
We 
defined $\text{El}(\mathcal S)$, $\text{GL}(\mathcal R)$  etc. 
from direct limits of square standard matrices,
with  square 
matrices  identified in the direct limit 
%  by inclusions into ``stabilized'' versions:
with their ``stabilized'' versions by maps  
  \begin{equation} \label{dlimitgenerators}
%    \begin{pmatrix} I-A   \end{pmatrix} =
 C\  \to \ 
 D=
 \begin{pmatrix} C & 0 \\ 0 & I 
  \end{pmatrix} \ .
%% \begin{pmatrix} I-A & 0 \\ 0 & I 
%% \end{pmatrix} \ .
  \end{equation}

  For matrices in $\MPE(\mathcal S)$, we generate
  the $\mathcal P$-partitioned direct limits
  $\text{El}_{\p}(\mathcal S)$, $\text{GL}_{\p}(\mathcal R)$ etc.  
 in the same way.
 Precisely,   the  
%  ``inclusion''    of
 %  matrices $C,D$ in  \eqref{dlimitgenerators}
map $C\to D$ in  \eqref{dlimitgenerators}
 %%  $\MPE(\mathcal S)$
 must come with
 an injection $\iota : \mathcal I^C \to \mathcal I^D$
 such that, after replacing each $i$ of $\mathcal I^C$ with $\iota (i)$,
 the following hold.
 \begin{enumerate}
   \item 
  $C$ is a principal submatrix of $D$. 
   \item   $\mathcal I^C_p = \mathcal I^D_p \cap \mathcal I^C_p$,
     for every $p$ in $\p$.
   \item
     % If      $\{i,j\} \nsubseteq \mathcal I^C$,
     If $\{i,j\}\subset \mathcal I^D$ and $\{i,j\} \nsubseteq \mathcal I^C$,
     then
     $D(i,j) = \delta_{ij} $. 
 \end{enumerate} 
  \end{de}
  Note, we could equally well display a $\mathcal P$ stabilization of $C$  
  using $D=\left(\begin{smallmatrix} I&0\\0&C\end{smallmatrix}\right)$;
    there is just a different
    tacit bijection between the index sets $\mathcal I^D_p$ and
    index sets for the displaying matrix.

    \iffalse 
\begin{rem} When we exhibit a standard matrix to present a 
  $\p$-partitioned matrix,
    the 
    $\p$-partitioning of indices need not be the standard one;
    a standard matrix presenting a square $\mathcal P$-partitioned matrix 
    is only determined up to
    conjugation by a permutation matrix. 
    So, in condition (2) 
    we could just as well have written 
    $\left(\begin{smallmatrix} I&0\\0&B\end{smallmatrix}\right)$
  in place of
  $\left(\begin{smallmatrix} B&0\\0&I\end{smallmatrix}\right)$. 
\end{rem}
\fi

\begin{rem} When we exhibit a standard matrix to present a 
  $\p$-partitioned matrix,
    the 
    $\p$-partitioning of indices need not be the standard one;
    a standard matrix presenting a square $\mathcal P$-partitioned matrix 
    is only determined up to
    conjugation by a permutation matrix. 
    So, in condition (2) 
    we could just as well have written 
    $\left(\begin{smallmatrix} I&0\\0&B\end{smallmatrix}\right)$
  in place of
  $\left(\begin{smallmatrix} B&0\\0&I\end{smallmatrix}\right)$. 
\end{rem}

  We can now give a $\mathcal P$-partitioned version of
  Theorem \ref{keyprop}.
  (See  Definitions \ref{slpnDefinition} and \ref{defSLPEquiv}
  for the $\text{SL}_{\p}$ terms in the statement of the theorem.)   
 
\begin{thm} \label{keyproppartitioned}
  Suppose $\mathcal R$ is a ring,
  $A\in \MPE(m,\mathcal R)$ and $B\in \MPE(n,\mathcal R)$.
  Let the index sets $\mathcal I^A, \mathcal I^B$ be disjoint. 
Suppose  there is a
  lag $\ell$ $\textnormal{SE}_{\p}-\mathcal R$ given by   
\[
A^{\ell} =RS,\quad B^{\ell}=SR, \quad AR = RB, \quad SA=BS \ .
\]
Then there is
% an
a $\p$-partitioned equation of matrices
\[
  \begin{pmatrix} I-tA & -R \\ 0 & I
  \end{pmatrix}
  \begin{pmatrix} I + (tA) + \cdots +(tA)^{\ell -1} & R \\
    -t^{\ell}S  & I-tB
  \end{pmatrix}
\     = \ 
    \begin{pmatrix} I& 0 \\ -t^{\ell}S  & I -tB 
    \end{pmatrix}
    \]
    in which the first $m$ rows/columns are indexed
    %like
    by
    $\mathcal I^A$; 
    the last $n$ rows/columns are indexed
    %like
    by
    $\mathcal I^B$; and
%    for each matrix $M$ in the equation,
%    $I^M_p
%    = \mathcal J^M_p= \mathcal I^A 
    %    \sqcup \mathcal I^B$.
       for each matrix $M$ in the equation,
    and each $p$ in $\mathcal P$,
    $I^M_p
    = \mathcal J^M_p= \mathcal I_p^A 
    \cup \mathcal I_p^B$. 

    Let $U$ be the intertwining matrix,
        $U =   \begin{pmatrix} I + (tA) + \cdots +(tA)^{\ell -1} & R \\
  -t^{\ell}S & I-tB  \end{pmatrix}      $; 
    then the following $\mathcal P$-partitioned
PSE (polynomial shift equivalence) equation holds: 
\begin{equation} \label{pseEquationPartitioned}
  \begin{pmatrix} I & 0 \\ t^{\ell}S & I
  \end{pmatrix}
  \begin{pmatrix} I & -R \\ 0 & I
  \end{pmatrix}
    \begin{pmatrix} I-tA & 0 \\ 0 & I
    \end{pmatrix}
    U
    \     = \ 
    \begin{pmatrix} I& 0 \\ 0  & I -tB 
    \end{pmatrix}
    \end{equation} 
If $\ell =1$, then $U$ 
lies in $El_{\mathcal P}(\mathcal R[t])$.
In any case, $U\in GL_{\p}(\mathcal R[t])$. 
If $\mathcal R$ is an integral domain, 
then
% $U\in SL(\mathcal R_{\p}[t])$.
$U\in SL_{\p}(\mathcal R[t])$.

In particular, if $\mathcal R = \Z$ then 
$U\in SL_{\p}(\Z [t])$, and the equation \eqref{pseEquationPartitioned} 
gives a stabilized $SL_{\p}(\Z[t])$-equivalence of
$(I-tA)$ and $(I-tB)$.

\end{thm}
\begin{proof}
The  claim that the exhibited equations are legal as $\p$-partitioned
equations should be clear, and the other proofs follow those
of  Theorem \ref{keyprop}. To verify the final claim that
$U\in \text{SL}_{\p}(\Z[t])$, 
note that if $M = M_1\cdots M_k$ is a product in $\mathcal M_{\mathcal P}$,
then for every $p$ in $\mathcal P$,
$M\{p,p\} = M_1\{p,p\}\cdots M_k\{ p,p\}$. 
So, the restriction of matrices in 
\eqref{pseEquationPartitioned} to their $\{p,p\}$ principal
submatrices gives an equation of the same form,
which by  Theorem \ref{keyprop} gives a stabilized
$\text{SL}_{\mathcal P}(\Z[t])$ equivalence of
  $(I-tA\{p,p\})$ and   $(I-tB\{p,p\})$. 
  \iffalse
note that for every
$p$ in $\p$,
we obtain from \eqref{pseEquationPartitioned}
another valid equation by restricting each of the five
$(m+n)\times (m+n)$ matrices in \eqref{pseEquationPartitioned}
to its $\{ p,p \}$ block.
\fi 
\end{proof}

\section{The flow equivalence classification} \label{febackground}

% A cycle component of a square nonnegative matrix is an irreducible
% component which is a permutation matrix.

 We recall first the  1983 result of Franks
 classifying
 irreducible SFTs up to flow equivalence. 
 (Franks by construction proved sufficiency of the
 algebraic invariants; necessity was established earlier
 by Parry and Sullivan  \cite{parrysullivan1975}
 and Bowen and Franks \cite{BowenFranks1977}.)
 An $n\times n$ matrix $M$ over $\Z$ defines a map
 $\Z^n \to \Z^n$, and 
by definition  the
 cokernel of $M$ is 
the quotient  $\Z^n /\text{Image}(M)$.
 \begin{footnote}
   {In general, even for a commutative  ring $\mathcal R$, the isomorphism
     class of the cokernel module can depend on whether
     $M$ acts on rows or columns. When $\mathcal R= \Z$,
     the row and column cokernel groups are isomorphic.}
   \end{footnote}

\begin{thm} \label{franks} \cite{franksFlowEq1984} 
  Suppose $A,B$ are irreducible matrices over $\Z_+$ which are
  not permutation matrices. Then the following are equivalent.
  \begin{enumerate}
\item 
    The SFTs $\sigma_A, \sigma_B$ are
    flow equivalent.
  \item The cokernel groups of $I-A$ and $I-B$ are isomorphic,
    and $\det(I-A) = \det(I-B)$. 
      \end{enumerate}
\end{thm} 

Suppose $A,B$ are irreducible matrices  over $\Z_+$ which
are not permutation matrices and which are
SE-$\Z_+$ (or even just SE-$\Z$).  
By setting $t$ equal to 1 in
 the PSE
 equation \eqref{pseEquation}, it is easy to deduce from Franks' theorem
 that the SFTs $\sigma_A, \sigma_B$ are flow
 equivalent\begin{footnote}{Before Franks' theorem, it was known that
     the shift equivalence class of $A$ determines
     $\det(I-tA)$
     (see \cite[Corollary 4.8, page 451]{WilliamsOneDim1970})
     and the isomorphism class of the cokernel group of 
     $(I-A)$ (see      \cite[Prop. 1.1]{BowenFranks1977}).
     Thus Franks' theorem showed shift equivalence implies
     flow equivalence in the  case $A$ is irreducible.
     This implication was implicit in the 
     1995 Lind-Marcus book 
 (see \cite[Theorem 6.4.6, Theorem 7.4.17, Corollary 7.4.12 and page
       456]{LindMarcus1995}) 
     and  explicit in the  
     1992 exposition \cite[Section 5.4]{BoyleMatrices1991}. The
     implication is also implicit in Kitchens' 1998 book
     (see \cite[Section 2]{KitchensBook1998}),
     if one adds the remark that $\det(I-tA)$ (and therefore
     $\det(I-A)$) is an invariant of shift equivalence.}\end{footnote}.  
Huang, after papers addressing special cases,
analogously gave a complete (unpublished) 
algebraic invariant for flow equivalence 
of general SFTs, involving
a  ``K-web'' of intertwined exact sequences of finitely generated
abelian groups. 
A different approach in
\cite{BPos2002}  
 characterized  
flow equivalence in terms of a certain
type of equivalence of matrices; 
the passage from that equivalence result to the
 K-web invariants was carried out in
\cite{BHuang2003}. To show SE-$\Z_+$  implies FE, 
we will not  take the path of arguing (analogous to the
appeal in the irreducible case 
to the complete algebraic invariants
% of Franks)
in Theorem \ref{franks})
that SE-$\Z_+$  induces an
appropriate isomorphism of $K$-webs.  
Rather, we will be using
Theorem \ref{feresult} below, 
 extracted from the classification theorem 
 \cite[Theorem 3.1]{BPos2002}. (For more about this ``extraction'', 
 see Remark
\ref{feResultNotationDifferences}.)

\begin{de} \label{slpnDefinition} Let $\p$ be a finite
  poset and $\mathcal R$ a commutative ring.
  Then $\text{SL}_{\p}(n, \mathcal R)$ is the
%  group of $n\times n$ square $\p$-partitioned matrices
set of $n\times n$ square $\p$-partitioned matrices
  over $\mathcal R$ for which every
  diagonal block has determinant 1. (The subset sharing a given
  $\mathcal P$-partitioning $\{\mathcal I_p\}$ of the 
  index set   is a group.) 
  \end{de} 

\begin{de} \label{defSLPEquiv}
  Let $\p$ be a finite
  poset and $\mathcal R$ a commutative ring.
  An  $\text{SL}_{\p}( \mathcal R)$ equivalence of
   $\p$-partitioned matrices $C,D$  
 is a 
 $\p$-partitioned matrix equation $UCV = D$
  with  $U$ and $V$ in $\text{SL}_{\p}( \mathcal R)$.
  A stabilized $\text{SL}_{\p}(\mathcal R)$ equivalence of
  $A$ and $B$ is an  $\text{SL}_{\p}(\mathcal R)$ equivalence of
  stabilizations
  $\left(\begin{smallmatrix} I-A&0\\0&I\end{smallmatrix}\right)$
  and
  $\left(\begin{smallmatrix} I-B&0\\0&I\end{smallmatrix}\right)$. 
  \end{de}

To discuss the flow equivalence classification of general SFTs, we need
to consider stabilizations, so we need to
consider matrices in $\ME(\Z_+)$ outside $\Mdelta(\Z_+) $.
\begin{de} 
$\MPdeltaess (\Z_+)$ 
is the set of
$A$ in   $\MPE(\Z_+)$ such that 
for all $p,q$ in $\mathcal P$ 
the following hold.   
\begin{enumerate} 
\item 
  $A\{ p,p\}$ is an essentially  irreducible component of $A$.
\item
 $p\preceq q$ in $\p$ if and only if 
   there exist $k>0$,  
   and indices $i$ and    $j$ for the
  irreducible components in $A\{p,p\}$ and $B\{q,q\}$, 
  such that $A^k(i,j) \neq 0$.
\end{enumerate}
\end{de}
{\raggedright If $A\in \MPdeltaess (\Z_+)$, 
then a $\mathcal P$-stabilization
$(I-A) \to  \left(\begin{smallmatrix} I-A&0\\ 0&I
\end{smallmatrix}\right)   := I-A'$
    produces $A'$ in $\MPdeltaess (\Z_+)$.}

    \begin{de} (Cycle components)
      \begin{enumerate}
      \item For $A$ in $\ME(\Z_+)$, 
        an irreducible component is a cycle component if  
        it is a cyclic permutation
        matrix.
      \item
        Suppose  $A\in \MPdeltaess (\Z_+)$ and $p\in \p$.
     We say $p$ is a cycle component of $A$ if 
               $A\{p,p\}$ is essentially  cyclic. 
        This terminology comes from the isomorphism of
        posets $\p \to \p^A$ given by the map which
        sends $p$ to the unique irreducible component of 
        $A\{p,p\}$.
        \end{enumerate} 
\end{de} 

%% Recall, for a square nonnegative matrix $A$, $\mathcal P^A$ denotes its
%% poset of irreducible components.
In the flow equivalence classification
for irreducible $A$,
permutation matrices were an exceptional case. 
Likewise for general $A$, 
the cycle components  require
special consideration.

For the next definition,  recall that when  $A$ is essentially
 cyclic, $\cok(I-A)$ is isomorphic to $\Z$ (this will be reproved
 a bit later). 

\begin{de} \label{PosOnCycleEquiv} Suppose $A$ and $B$ are essentially cyclic
  square matrices over $\Z_+$, and $(I-B)= U(I-A)V$ is
  an $\text{SL}(n, \Z)$ equivalence. The induced 
isomorphism of cokernels $\Z^n/(I-A)\Z^n \to 
  \Z^n/(I-B)\Z^n$ given by $[v]\mapsto [Uv]$ is 
  positive 
 if for  canonical basis vectors
  $e$ and $e'$ generating the cokernels we have $[Ue]=[e']$ in
  $\Z^n/(I-B)\Z^n$.
  \end{de} 

Matrices  $A,B$  in $\MPdeltaess(\Z_+)$ have 
 the same cycle components if for every $p$ in $\p$,
 $A\{ p,p \}$ is essentially cyclic if and only if 
 $B\{ p,p \}$ is essentially cyclic.
 In this case,   an $\slpnz $ equivalence of $I-A$ and
 $I-B$ is  ``positive on cycle components'' 
  if for every   cycle component $p$, 
 the  isomorphism $\cok((I-A)\{p,p\}) \to
  \cok((I-B)\{p,p\})$ 
  induced by  
  $(I-B)\{p,p\}  = U\{p,p\}(I-A)\{p,p\}V\{p,p\}$
  is a positive
  isomorphism.

  By Proposition \ref{diagblocksirred}, every SFT is topologically
  conjugate to some $\sigma_A$ such that    $A\in \MEdelta(\Z_+)$.
  So, to show shift equivalence implies flow equivalence,
it will suffice to consider matrices  in $ \MEdelta(\Z_+)$.
\begin{thm}\label{feresult}
    Suppose $A,B$ are square matrices in $\MEdelta(\Z_+)$.  Then
 the following are
  equivalent.
  \begin{enumerate}
\item 
    The SFTs $\sigma_A, \sigma_B$ are
    flow equivalent.
  \item There is a poset $\p$, and $\p$-partitions for $A$ and $B$
    with which  $A$ and $B$ are in $\MPdeltaess(\Z_+)$,  
    such that
    \begin{enumerate}
    \item $A$ and $B$ have the same cycle components, and
      \item 
    there are 
 $\mathcal  P$ stabilizations
    $\left(\begin{smallmatrix} A&0\\0&I\end{smallmatrix}\right)$,
    $\left(\begin{smallmatrix} B&0\\0&I\end{smallmatrix}\right)$
        which are $\slpnz $ equivalent, by an equivalence which
        is positive on cycle components.
    \end{enumerate}
      \end{enumerate}
\end{thm}

The next result is the last statement of this section needed
in Section \ref{secMainResultProof}
for the proof of Theorem \ref{maintheorem} (and it is only needed
in the case $A$ and $B$ have cycle components). 
\begin{prop}\label{esscyclicpos}
  Suppose $A,A'$ are essentially
  cyclic matrices and $(U,V)$ gives an $\text{SL}(n,\Z)$ equivalence between 
  $I-A$ and $I-A' =U(I-A)V$. If $U$  is nonnegative, then
  the induced  cokernel isomorphism
  $[v]\mapsto [Uv]$ is positive. 
\end{prop}

\iffalse
The rest of this section works to  the proof of
Proposition \ref{esscyclicpos}, and provides some
related material

We turn now to more explanation of the 
 ``positive on cycle components'' condition. 
The next two paragraphs establish terminology for
Propositon \ref{esscyclic}. 
\fi

We now develop some notation used in the proof of a lemma, and
then the proof of Proposition \ref{esscyclicpos}. 
Consider a   matrix $A$ over $\Z_+$ which is
essentially cyclic (i.e., there is a unique maximal irreducible
component, and it is a permutation matrix). 
After passing to a matrix conjugate to $A$ by a permutation,
we have standard matrices with the block forms 
\begin{align*} 
A & = \begin{pmatrix} N_1 & X & Y \\ 0 & P & Z \\ 0 & 0 & N_2
\end{pmatrix} \\
I-A &= \begin{pmatrix} I-N_1 & -X & -Y \\ 0 & I-P & -Z \\ 0 & 0 & I-N_2
\end{pmatrix}
\end{align*} 
with $N_1, N_2 $ nonnegative nilpotent, and zero on and below their diagonals,
and $P$ a cyclic permutation matrix.
(It can also happen that the first and/or third blocks do not arise. 
 I will write for the displayed form; 
arguments are easily adjusted for the other cases.)

Let $A$ as above be $n\times n$, with $n =a+b+c$, where 
$N_1$ is $a \times a$,  $P$ is $b\times b$ and $N_2$ is
$c\times c$. Let matrices act on column vectors
with integer entries. For typographical ease, we will write a column
vector $v$ in $\Z^n$ as $v= (x,y,z)$ with $x \in \Z^a$, $y\in \Z^b$
and $z\in \Z^c$.
For a column vector
$v$ in $\Z^n$, $[v]$ is the corresponding element of the
cokernel group $G=\Z^n/(I-A)\Z^n$. When we say $v=v'$ in $G$,
we mean $[v] = [v']$.

\begin{lem}\label{esscyclic} (We use the terminology of the
  last two paragraphs.) 
  There exists a standard basis vector $e$ such that 
  the cokernel group $G=\Z^n/(I-A)\Z^n$ equals $\Z[e]$,
  isomorphic to $\Z$. 
  There are nonnegative integers $n_1, \dots , n_c$ such that for
  all $(x,y,z)$, 
  \begin{equation} \label{poscoefficients}
    [ (x,y,z)] =( y_1 + \dots + y_b +n_1 z_1 + \dots +
n_c z_c)[e] \ .
\end{equation} 
\end{lem}
\begin{proof}
  Every $(x,0,0)$ is trivial in $G$, because $(I-N_1)$ is invertible,
  and if $x' = (I-N_1)^{-1}x$ then 
  $I-A: (x',0,0)\to (x,0,0)$.

  If $e_i$ is a canonical basis vector for $\Z^b$, then
  $(I-A)e_i$ has the form $(v,e_i -Pe_i,0)$, which equals
  $(0,e_i-Pe_i, 0)$ in $G$. Because $P$ is a cyclic permutation,
  we conclude that for any two  canonical basis vectors $e_i, e_j$
  of $\Z^b$, $(0,e_i,0) = (0,e_j,0)$ in $G$. Fix some canonical basis
  vector $e_i$ of $\Z^b$ and set $e=(0,e_i,0)$. 
  Then for 
  $y$ in $\Z^b$, 
  we have $(0,y,0) =  Me$ in $G$, where
  $M = \sum_{1\leq i \leq b} y_b$. 

  Now let $\ell > 0$ be such that  $(N_2)^{\ell +1} =0$.
  Then $(I-N_2)$ is invertible, with inverse
  $U = I+N_2 + \cdots +(N_2)^{\ell}$, which is a nonnegative
  matrix. Suppose $e_k$ is a canonical basis vector for $\Z^c$. 
Then 
$I-A: (0,0,Ue_k) \mapsto (-YUe_k, -ZUe_k, e_k)$ 
and therefore  $(0, 0, e_k) = (0, ZUe_k, 0)$ in $G$. 
Let $n_k$ be the sum of the entries of the nonnegative vector
$ZUe_k$; then  $e_k = n_k e$ in $G$. 

This shows that $[e]$ 
is a generator of $G$, and \eqref{poscoefficients} holds. 
Because $\det(I-A)=0$, the group $G$ is infinite.
Therefore $\Z [e]$ is isomorphic to  $\Z$. 
  \end{proof}

\iffalse
When $A$ is essentially cyclic over $\Z_+$,
we make $\cok (I-A)$ an ordered group by
declaring the positive set to be $\Z_+ [e]$, where
$e$ is a canonical basis vector such that
$\cok (I-A) = \Z[e]$. 
An $\text{SL}(n, \Z)$-equivalence 
  $(I-B) = U(I-A)V$ induces 
 an isomorphism
$ \cok (I-A) \to \cok (I-B)$,
 by the rule $[v] \mapsto [Uv]$.
So, in Proposition \ref{esscyclic},
the equivalence being positive on cycles means that
for the essentially cyclic diagonal blocks, 
the  equivalences $(I-B)\{p,p\} = U\{p,p\}(I-A)\{p,p\}V\{p,p\}$
  must induce ordered group isomorphisms
  $\cok  (I-A)\{p,p\} \to \cok(I-B)\{p,p\} $. 
\fi

\begin{proof}[Proof of Prop. \ref{esscyclicpos}]
  Suppose $e,e'$ are canonical basis vectors
  such that $[e]$ and $[e'] $  are generators
  of the  cokernels of $(I-A)$ and $(I-A')$. 
  In the notation of  Lemma \ref{esscyclic},
  let $Ue_i= (x,y,z)$, the $i$th column of $U$.
  There are nonnegative integers $n_j$
  such that 
  $[Ue] = (\sum_i y_i  + \sum_j n_j z_j)[e']$ in
  $\cok (I-A')$.   Because $U$ induces an isomorphism
  of the cokernels, 
$\sum_i y_i  + \sum_j n_j z_j$ must be 1 or -1. Because 
$U$  is nonnegative, the sum is 1.
     \end{proof}

\begin{rem}
  Of course we could have worked with row cokernels in place of
  column cokernels. 
  It is interesting that one only needs to track the
  positive-on-cycles condition on one side, not for
  both row and column cokernels.
\end{rem}

\begin{rem}\label{minorErratumCorrected}
  Here we correct a minor erratum in  \cite{BPos2002}.

Lemma 2.6 of \cite{BPos2002} is a version of 
Lemma \ref{esscyclic}.
 Lemma 2.6 contains an error: 
Lemma 2.6(3)
says (in the language of Lemma \ref{esscyclic})
that $[(e_i,0,0)] = [(0,0,e_k)] = 0$, i.e.
$ [(x,y,z)] =( \sum_i y_i)[e]$.
Nevertheless, after replacing Lemma 2.6
with Lemma \ref{esscyclic} above, 
the proofs in 
\cite{BPos2002} 
relying on Lemma 2.6
 go through.
Lemma 2.6 is used at two points in \cite{BPos2002}:  
\begin{enumerate}
\item In the proof of  \cite[Proposition 2.7(2)]{BPos2002},
  where our Proposition \ref{esscyclicpos} applies. 
  
  \item 
    In  \cite[Lemma A.3,line 5]{BPos2002},
    to show $a=1$. This follows easily from
    Lemma \ref{esscyclic}.
\end{enumerate}
\end{rem}

\begin{rem} \label{feResultNotationDifferences}
  Theorem \ref{feresult} is obtained from the equivalence of (1) and (3)
  in \cite[Theorem 3.1]{BPos2002}, as indicated below. 
%  notational differences in the two statements. 
  \begin{enumerate}

  \item
    It is  natural to visualize elements of the 
    stabilized general linear group 
    $\text{GL} (\mathcal R)$ as infinite ``matrices''
     equal to the identity in all but finitely many
     entries
     (e.g. \cite[p. 522]{VasilovStepanov2011}). 
      In this spirit, in \cite{BPos2002}  elements of the direct limit group
  $\text{SL}_{\mathcal P} (\Z)$ were formally described using
  ``matrices'' with infinite index sets (equal to the identity
  matrix in all but finitely many entries);  
   correspondingly, 
  matrices  presenting SFTs were allowed to have infinite
  index sets, and required to be nonzero in only finitely many entries.
%Now,  I think this  intuitive viewpoint 
% was a nonoptimal choice for formal notation. 
  In the current paper, we follow a more standard convention,
  and express results only using
 matrices with finite index sets. 
\item
  A matrix $A$ in $\MPdeltaess(\Z_+)$ can be presented as a
  standard matrix, with the sets $\mathcal I^A_p$ being successive
  integers (so, a submatrix $A\{p,q\}$ is a block in the
  block triangular structure). One can check that
  the $\p$-blocked standard matrices derived in this way from
  matrices in   $\MPdeltaess(\Z_+)$  are
  (apart from the
infinite vs. finite matrix difference discussed above)
the matrices in  the class
$\mathfrak M^o_{\mathcal P,+}(\Z)$ which appears in
the statement of 
\cite[Theorem 3.1]{BPos2002}. 
\item
In \cite[Theorem 3.1]{BPos2002}, 
for matrices $A,A'$ in $\mathfrak M^o_{\mathcal P,+}(\Z),
\mathfrak M^o_{\mathcal P',+}(\Z)$
(respectively) 
there is a poset isomorphism $\nu: \mathcal P \to \mathcal P'$
inducing a permutation $P$ of indices such that
there is a positive stabilized $\text{SL}_{\mathcal P}(\Z)$
equivalence of $(I-A)$ and $I-(P^{-1}A'P)$.
In our current notation, this replaces some chosen
$\mathcal P'$ partitioning $\{ \mathcal I^B_q: q\in \mathcal P'\}$
with a $\mathcal P$-partitioning
$\{ \mathcal I^B_{\nu(p)}: p\in  \mathcal P\}$.
\item
  If   $A\in \MEdelta(\Z_+)$, then for $\p = \p^A$, with the
  natural $\p$ partitioning of indices we have 
  $A\in \MPdelta(\Z_+)$. If $I-A'$ is any $\p$-stabilization of
  $I-A$, 
  then $A'\in \MPdeltaess(\Z_+)$. Thus
  \cite[Theorem 3.1]{BPos2002} applies to the matrices
  considered in Theorem \ref{feresult}. 
  \end{enumerate}
  \end{rem}

\section{SE implies FE for SFTs} \label{secMainResultProof}

In this section, we give the proof for Theorem \ref{maintheorem}. 

\begin{proof}[Proof of Theorem \ref{maintheorem}]
  Suppose $A$ and $B$ are standard nonnilpotent matrices,
  $m\times m$ and $n\times n$, 
and there is a shift equivalence over $\Z_+$, 
  \begin{equation}\label{seEQN}
A^{\ell} =RS,\quad B^{\ell}=SR, \quad AR = RB, \quad SA=BS \ .  
\end{equation} 
We will show that $\sigma_A$ and $\sigma_B$ are flow equivalent.
Because topologically conjugate SFTs are (obviously) flow equivalent,
it follows from Proposition \ref{diagblocksirred} that without
loss of generality we may assume
$A$ and $B$ are block upper triangular with every diagonal block
an irreducible matrix. 
Then by Proposition \ref{fromSEtoSEP}, 
with $\mathcal P=\mathcal P^A$, there is a 
 unique choice of 
 $\p$-partitions for 
$B,R,S$ such that the equation \eqref{seEQN} defines
an $SE_{\p }-\Z_+$. 
By Theorem \ref{keyproppartitioned}, we then get the
$\p$-partitioned PSE equation 
\begin{equation} %\label{pseEquationPartitioned}
  \begin{pmatrix} I & 0 \\ t^{\ell}S & I
  \end{pmatrix}
  \begin{pmatrix} I & -R \\ 0 & I
  \end{pmatrix}
    \begin{pmatrix} I-tA & 0 \\ 0 & I
    \end{pmatrix}
    U
    \     = \ 
    \begin{pmatrix} I& 0 \\ 0  & I -tB 
    \end{pmatrix}
    \end{equation} 
with 
$U \in SL_{\p}(m+n, \Z [t])$.
Setting $t=1$, we obtain an
$\text{SL}_{\p}(m+n, \Z) $ equivalence of
stabilizations of $I-A$ and $I-B$: 

\begin{equation} \label{stabilizedSLequivalenceat1}
  \begin{pmatrix} I & 0 \\ S & I
  \end{pmatrix}
  \begin{pmatrix} I & -R \\ 0 & I
  \end{pmatrix}
    \begin{pmatrix} I-A & 0 \\ 0 & I
    \end{pmatrix}
\begin{pmatrix}  W & R \\
  -S & I-B  \end{pmatrix} 
    \     = \ 
    \begin{pmatrix} I& 0 \\ 0  & I- B 
    \end{pmatrix}
    \end{equation} 
in which $W=  I+ A + \cdots +A^{\ell -1} $ and 
the determinant of $\begin{pmatrix} W & R \\
  -S & I-B  \end{pmatrix}$ is 1.

Consequently, it will follow from  
Theorem \ref{feresult} that $\sigma_A$ and $\sigma_B$ are
flow equivalent, if we can verify that
the equivalence \eqref{stabilizedSLequivalenceat1}
is positive on cycle components.
So, suppose $ p$ is a cycle component. For a $\p$-partitioned
matrix $E$, let $E_p$ denote $E\{ p,p \}$.
The equivalence \eqref{stabilizedSLequivalenceat1} induces an isomorphism
$\cok\left(\begin{smallmatrix} I-A&0\\0&I\end{smallmatrix}\right)_p
  \to
  \cok\left(\begin{smallmatrix} I&0\\0&I-B\end{smallmatrix}\right)_p
$ 
\iffalse
\[
  \cok\begin{pmatrix}I-tA&0\\0&I \end{pmatrix}_p
  \]
  \fi
  defined by the rule $[v]\mapsto [Mv]$, with 
 $M = 
  \left(\begin{smallmatrix} I&-R\\0&I\end{smallmatrix}\right)_p
\left(\begin{smallmatrix} I&0\\S&0\end{smallmatrix}\right)_p$. 
  Let $A_p$ be $a\times a$, let $B_p$ be $b\times b$, and write
  a vector in $\Z^{a+b}$ as $\left(\begin{smallmatrix} x \\ y
    \end{smallmatrix}\right)$ where $x\in \Z^a$ and $y\in \Z^b$. 
  \iffalse
&\to
  \cok\begin{pmatrix} I&0\\0&I-tB \end{pmatrix}_p \\
  v & \mapsto \begin{pmatrix}I&0\\S&I \end{pmatrix}_p
  \begin{pmatrix}I&-R\\0&I \end{pmatrix}_p v
  \end{align*}
\fi
The map $\left(\begin{smallmatrix} x \\ y
\end{smallmatrix}\right) \to
\left(\begin{smallmatrix} I & -R\\ 0&I 
    \end{smallmatrix}\right)
\left(\begin{smallmatrix} x \\ y
\end{smallmatrix}\right)$ induces an isomorphism
$\phi_1: \cok\left(\begin{smallmatrix} I-A&0\\0&I\end{smallmatrix}\right)_p
  \to
  \cok\left(\begin{smallmatrix} I-A&-R\\0&I\end{smallmatrix}\right)_p
    $
with $[\left(\begin{smallmatrix} x \\ y
\end{smallmatrix}\right)] \mapsto 
[\left(\begin{smallmatrix} x  -Ry \\ y
  \end{smallmatrix}\right)]$.
Because 
$[\left(\begin{smallmatrix} x \\ y
  \end{smallmatrix}\right)]
=
[\left(\begin{smallmatrix} x \\ 0
  \end{smallmatrix}\right)]
$ in 
$ \cok\left(\begin{smallmatrix} I-A&0\\0&I\end{smallmatrix}\right)_p$,
  and $\phi_1: [\left(\begin{smallmatrix} x \\ 0
\end{smallmatrix}\right)] \mapsto 
[\left(\begin{smallmatrix} x   \\ 0
  \end{smallmatrix}\right)]$, 
  the isomorphism $\phi_1$ is certainly positive.
  Let $C = \left(\begin{smallmatrix} A&R\\0&0\end{smallmatrix}\right)$.
    Then $C_p$ is an essentially cyclic matrix over $\Z_+$,
    and the map $v\mapsto 
    \left(\begin{smallmatrix} I&0\\S&I\end{smallmatrix}\right)_pv $
      induces an isomorphism
      $\phi_2:
      \cok\left(\begin{smallmatrix} I-A&-R\\0&I\end{smallmatrix}\right)_p
        \to 
        \cok\left(\begin{smallmatrix} I&0\\0&I-B\end{smallmatrix}\right)_p $.
          Because
          $\left(\begin{smallmatrix} I&0\\S&I\end{smallmatrix}\right)_p $
is nonnegative, it follows from             
Proposition \ref{esscyclicpos} that
$\phi_2 $ is a positive isomorphism. Therefore $\phi$, as the composition
of positive isomorphisms, is a positive isomorphism. 
  This finishes the proof.
\end{proof} 

\begin{rem} One reality check for the proof is
  to consider the ``positive on cycles'' condition to be defined for
  row vectors, and check that 
  the multiplication on row vectors by
$\left(\begin{smallmatrix} W& R\\-S &I-B\end{smallmatrix}\right)$
  induces an equivalence positive on cycles. This does 
  hold,  because the matrix  $R$ is nonnegative.
\end{rem}

\begin{rem} \label{brixetc} Suppose matrices $A,B$ are in $\MPE(\Z_+)$.
  It can happen that they are   $\text{SE}_{\p}-\Z$,
 with the same cycle components,
  but still define $\sigma_A$, $\sigma_B$
  which are not flow equivalent -- so, one possible proof strategy
  for showing shift equivalence implies flow equivalence cannot work. 
  (I thank Kevin Brix for pointing this out to me.)
  However, from the proof above of Theorem \ref{maintheorem}, 
  we see that
if $A$ and $B$ are $\text{SE}_{\p}-\Z$ and have no cycle components,
  then  $\sigma_A$ and $\sigma_B$ are flow equivalent. 
  \end{rem}

\begin{qu} 
Is there  an example of two subshifts which are eventually
conjugate (or even just eventually flow equivalent) and are
not flow equivalent?
\end{qu} 

\section{Eventually conjugate systems which are not flow equivalent}
\label{secExample}

For completeness, we will construct a (rather degenerate)
example of 
systems (not subshifts) which are eventually conjugate but are
not flow equivalent. To clarify the main idea,  suppose $S_m$ is the union of
$m$ fixed points and for all $k >1$ countably many orbits of size $k$,
all accumulating on a single fixed point.
For each $k>1$, the system $(S_m)^k$ contains countably many orbits
of every finite size, all accumulating on a fixed point.
The systems $S_m$ are eventually conjugate but
not conjugate.  We have to complicate this
construction a bit to get eventually conjugate systems which also
are not flow equivalent. 

Given a countably infinite collection 
$\{ (X_n, T_n): n\in \N\}$  of systems 
such that each $T_n$ has an infinite dense orbit, 
let $X^*$ be the one point compactification of the
disjoint union of the 
$X_n$, $X^* = \{ \infty\} \cup (\sqcup_n X_n)$. We define the
 one point compactification $T^*: X^* \to X^*$ of the
 $T_n$ by letting
 $\infty$ be a fixed point of $T^*$.
 Let $Y^*$ be the mapping torus over $T^*$ and let $Y_n$ be the
 mapping torus over $T_n$.

 We claim that $\{Y_n: n\in \N\}$ is the collection of maximal closed connected
 subsets of $Y^*$ which are not circles. To see this, suppose
 $n\in \N$,  and note the following.  
 \begin{enumerate}
 \item
   A closed connected subset of $Y^*$ is either contained in $Y_n$ 
   or is disjoint from $Y_n$.
   \item 
     $Y_n$ is a closed connected subset of $Y^*$ (because the dense
     infinite orbit      of $T_n$ produces a dense immersed line
     in $Y_n$).
   \item
     $Y_n$ is not a circle.
   \item
     $Y^*$ is the union of the $Y_n$ and the circle over $\{\infty\}$.
 \end{enumerate}
Now suppose $((X')^*,(T')^*)$ is the one point compactification
of another such collection $(X'_n, T'_n)$.
Let $Y'_n,(Y')^*$ denote the associated mapping tori. 
Suppose $\phi: Y^* \to (Y')^*$ is an orientation preserving
 homeomorphism  giving a flow equivalence of
 $T^*$ and $(T')^*$. From the claim, we deduce
 that there must be a bijection $\alpha: \N \to \N$
 such that for every $n$, $\phi (Y_n) = Y'_{\alpha(n)} $,
   so $T_n$ is flow equivalent to $T'_{\alpha(n)}$.

 Given a system $(X,T)$ and $h$ in $\N$, the height $h$
discrete suspension of $T$ is the homeomorphism
\begin{align*}
  T_h  : X \times \{0, \dots , h-1\}
  &\to X \times \{0, \dots , h-1\}\\
  (x,i) & \mapsto (x,i+1) \quad \text{ if } i<h-1 \\
  (x,h-1) &\mapsto (Tx, 0)\ .
\end{align*}
For all $h$, $T_h$ is flow equivalent to $T$.

Now let $(X,T)$ be a system (for example, the 2-shift)
with a dense orbit, such that
for $k>1$, $T$ is not flow equivalent to $T^k$. Let
$T_{\ell, h}$ denote $(T^{\ell})_h$.
For $m$ in $\N$, let $T^*_{(m)}$ be the one point compactification
of $m$ disjoint
copies of $T$ together with,
for all $\ell > 1$ and $h\geq 1$, a countably infinite collection of  copies
of $T_{\ell , h}$. No
such $T_{\ell , h}$  is
flow equivalent to $T$. So, if $m\neq n$, then
$T^*_{(m)}$ is not flow equivalent to $T^*_{(n)}$.

Suppose $k>1$. Then,  $(T^*_{(m)})^k$ is the one point compactification
of $m$ copies of $T^k$ together with the $k$th powers of the 
countably infinite collection of copies
of $T_{\ell, h}$, for  $\ell > 1$ and $h\geq 1$.
  Because $(T_{k,k})^k$ is the disjoint union of
  $k$ copies of $T^k$, 
$(T^*_{(m)})^k$ 
  contains
  countably many copies of $T^k$, regardless of $m$.
Thus for $k>1$, all the $(T^*_{(m)})^k$ are
conjugate. So, if $m\neq n$, then
$T^*_{(m)}$ is eventually conjugate
but not flow equivalent to $T^*_{(n)}$.
\iffalse
For what it's worth, the interested reader can 
 check that for every $k>1$ and
  every $m$, the system $(T^*_{(m)})^k$ is isomorphic to
  the one point compactification of
  the union over $\ell >1, h\geq 1$ of 
  countably many copies of
$ T_{\ell, h}$.
\fi 
%%%   A proof appeals to two facts. 
%%%   (i) for $K>1$, 
%%%   $(T_{\ell, hk})^k$ is a disjoint union of $k$ copies of
%%%   $(T_{\ell, h})^k$.
%%%   (ii) If  $d= \gcd (h,k)$, with $h=ad$ and $k=bd$, then $(T_{h,\ell})^k$
%%%   consists of $d$ copies of $T_{a,b\ell}$. 
  
\bibliographystyle{plain}

\bibliography{reducible}

\begin{thebibliography}{10}

\bibitem{BowenFranks1977}
Rufus Bowen and John Franks.
\newblock Homology for zero-dimensional nonwandering sets.
\newblock {\em Ann. of Math. (2)}, 106(1):73--92, 1977.

\bibitem{BW04}
M.~Boyle and J.~B. Wagoner.
\newblock Positive algebraic {$K$}-theory and shifts of finite type.
\newblock In {\em Modern dynamical systems and applications}, pages 45--66.
  Cambridge Univ. Press, Cambridge, 2004.

\bibitem{BoyleJordanForm1984}
Mike Boyle.
\newblock Shift equivalence and the {J}ordan form away from zero.
\newblock {\em Ergodic Theory Dynam. Systems}, 4(3):367--379, 1984.

\bibitem{BoyleMatrices1991}
Mike Boyle.
\newblock Symbolic dynamics and matrices.
\newblock In {\em Combinatorial and graph-theoretical problems in linear
  algebra ({M}inneapolis, {MN}, 1991)}, volume~50 of {\em IMA Vol. Math.
  Appl.}, pages 1--38. Springer, New York, 1993.

\bibitem{BAlgAspects2000}
Mike Boyle.
\newblock Algebraic aspects of symbolic dynamics.
\newblock In {\em Topics in symbolic dynamics and applications ({T}emuco,
  1997)}, volume 279 of {\em London Math. Soc. Lecture Note Ser.}, pages
  57--88. Cambridge Univ. Press, Cambridge, 2000.

\bibitem{BPos2002}
Mike Boyle.
\newblock Flow equivalence of shifts of finite type via positive
  factorizations.
\newblock {\em Pacific J. Math.}, 204(2):273--317, 2002.

\bibitem{bce:fei}
Mike Boyle, Toke~Meier Carlsen, and S{\o}ren Eilers.
\newblock Flow equivalence and isotopy for subshifts.
\newblock {\em Dyn. Syst.}, 32(3):305--325, 2017.
\newblock Corrigendum: ibid, page (ii).

\bibitem{BHuang2003}
Mike Boyle and Danrun Huang.
\newblock Poset block equivalence of integral matrices.
\newblock {\em Trans. Amer. Math. Soc.}, 355(10):3861--3886, 2003.

\bibitem{BoSc1}
Mike Boyle and Scott Schmieding.
\newblock Strong shift equivalence and algebraic {K}-theory.
\newblock {\em J. Reine Angew. Math.}, 752:63--104, 2019.

\bibitem{BSCHStableAlg}
Mike Boyle and Scott Schmieding.
\newblock Symbolic dynamics and the stable algebra of matrices.
\newblock In Peter J.~Cameron R.A.Bailey and Yaokun Wu, editors, {\em Groups
  and graphs, designs and dynamics}, volume 491 of {\em London Mathematical
  Society Lecture Note Series}, pages 266--422. Cambridge University Press,
  2024.

\bibitem{cuntzkrieger}
Joachim Cuntz and Wolfgang Krieger.
\newblock A class of {$C\sp{\ast} $}-algebras and topological {M}arkov chains.
\newblock {\em Invent. Math.}, 56(3):251--268, 1980.

\bibitem{EilersReport2022}
S{\o}ren Eilers.
\newblock The {W}illiams problem through the lens of {C}untz-{K}rieger
  algebras.
\newblock {\em Oberwolfach Rep.}, 19(3):2111--2113, 2022.
\newblock Abstracts from the workshop held August 7--13, 2022, Organized by
  Dimitri Shlyakhtenko, Andreas Thom, Stefaan Vaes and Wilhelm Winter.

\bibitem{errs:complete}
S{\o}ren Eilers, Gunnar Restorff, Efren Ruiz, and Adam P.~W. S{\o}rensen.
\newblock The complete classification of unital graph {$C^*$}-algebras:
  geometric and strong.
\newblock {\em Duke Math. J.}, 170(11):2421--2517, 2021.

\bibitem{Fitting1936}
Hans Fitting.
\newblock \"{U}ber den {Z}usammenhang zwischen dem {B}egriff der
  {G}leichartigkeit zweier {I}deale und dem \"{A}quivalenzbegriff der
  {E}lementarteilertheorie.
\newblock {\em Math. Ann.}, 112(1):572--582, 1936.

\bibitem{franksFlowEq1984}
John Franks.
\newblock Flow equivalence of subshifts of finite type.
\newblock {\em Ergodic Theory Dynam. Systems}, 4(1):53--66, 1984.

\bibitem{S21}
K.~H. Kim and F.~W. Roush.
\newblock Williams's conjecture is false for reducible subshifts.
\newblock {\em J. Amer. Math. Soc.}, 5(1):213--215, 1992.

\bibitem{S11}
K.~H. Kim and F.~W. Roush.
\newblock The {W}illiams conjecture is false for irreducible subshifts.
\newblock {\em Ann. of Math. (2)}, 149(2):545--558, 1999.

\bibitem{KitchensBook1998}
Bruce~P. Kitchens.
\newblock {\em Symbolic {D}ynamics}.
\newblock Universitext. Springer-Verlag, Berlin, 1998.
\newblock One-sided, two-sided and countable state Markov shifts.

\bibitem{LindMarcus1995}
Douglas Lind and Brian Marcus.
\newblock {\em An introduction to symbolic dynamics and coding}.
\newblock Cambridge University Press, Cambridge, 1995.

\bibitem{LindMarcus2021}
Douglas Lind and Brian Marcus.
\newblock {\em An {I}ntroduction to {S}ymbolic {D}ynamics and {C}oding}.
\newblock Cambridge University Press, Cambridge, second edition, 2021.

\bibitem{parrysullivan1975}
Bill Parry and Dennis Sullivan.
\newblock A topological invariant of flows on {$1$}-dimensional spaces.
\newblock {\em Topology}, 14(4):297--299, 1975.

\bibitem{VasilovStepanov2011}
N.~A. Vavilov and A.~V. Stepanov.
\newblock Linear groups over general rings. {I}. {G}eneralities.
\newblock {\em Zap. Nauchn. Sem. S.-Peterburg. Otdel. Mat. Inst. Steklov.
  (POMI)}, 394:33--139, 295, 2011.

\bibitem{Warfield1978}
R.~B. Warfield, Jr.
\newblock Stable equivalence of matrices and resolutions.
\newblock {\em Comm. Algebra}, 6(17):1811--1828, 1978.

\bibitem{WilliamsOneDim1970}
R.~F. Williams.
\newblock Classification of one dimensional attractors.
\newblock In {\em Global {A}nalysis ({P}roc. {S}ympos. {P}ure {M}ath., {V}ols.
  {XIV}, {XV}, {XVI}, {B}erkeley, {C}alif., 1968)}, volume XIV-XVI of {\em
  Proc. Sympos. Pure Math.}, pages 341--361. Amer. Math. Soc., Providence, RI,
  1970.

\end{thebibliography}

\end{document}